\documentclass[11pt]{article}
\usepackage{amsmath,amstext,amssymb,amsthm, amsfonts}
\usepackage[title]{appendix}
\usepackage{cite}
\usepackage{color}
\newtheorem{theorem}{Theorem}[section]
\newtheorem{lemma}[theorem]{Lemma}

\newtheorem{definition}[theorem]{Definition}
\newtheorem{assumption}[theorem]{Assumption}
\newtheorem{example}[theorem]{Example}
\newtheorem{proposition}[theorem]{Proposition}
\newtheorem{corollary}[theorem]{Corollary}

\theoremstyle{remark}
\newtheorem{remark}[theorem]{Remark}

\numberwithin{equation}{section}

\def\slim{\mathop{\rm lim\,sup}}
\def\ilim{\mathop{\rm lim\,inf}}
\def\R{\mathbf{R}}
\def\olR{\overline{\R}}
\def\SS{\mathbb{S}} 		
\def\SSig{\Sigma} 		
\def\M{\mathcal{M}} 		
\def\h{\mathbf{I}}		
\def\S{\mathbb{S}}				
\def\P{\mathbf{P}} 		
\def\B{\mathcal{B}}
\def\ee{\varepsilon}

\begin{document}

\title{Fatou's Lemma for Weakly Converging Measures under the Uniform Integrability Condition}

\date{}

\maketitle

\begin{center}
Eugene~A.~Feinberg\footnote{Department of Applied Mathematics and
Statistics,
 Stony Brook University,
Stony Brook, NY 11794-3600, USA, eugene.feinberg@sunysb.edu},\
Pavlo~O.~Kasyanov\footnote{Institute for Applied System Analysis,
National Technical University of Ukraine ``Igor Sikorsky Kyiv Polytechnic
Institute'', Peremogy ave., 37, build, 35, 03056, Kyiv, Ukraine,\
kasyanov@i.ua.},\ and Yan~Liang\footnote{Department of Applied Mathematics and
Statistics,
 Stony Brook University,
Stony Brook, NY 11794-3600, USA, yan.liang@stonybrook.edu}
\end{center}

\begin{abstract}
 This note describes Fatou's lemma 
 for a sequence of measures converging weakly to a finite measure and for a sequence of functions  whose negative parts are uniformly integrable  with respect to these measures.  The note also provides new formulations of uniform Fatou's lemma, uniform Lebesgue convergence theorem, the Dunford-Pettis theorem, and the fundamental theorem for Young measures based on the equivalence of uniform integrability and the apparently weaker property of asymptotic uniform integrability  for sequences of functions and finite measures.
\end{abstract}

\maketitle


\section{Introduction}\label{ZKsec:1}
Fatou's lemma states under appropriate conditions that the integral of the lower limit
of a sequence of functions is not greater than the lower limit of the integrals.
This inequality holds under one of the following conditions:
{\rm (i)} each function is nonnegative; {\rm (ii)} the sequence of functions is bounded below by an  integrable function;
{\rm (iii)} the sequence of negative parts of the functions is  uniformly integrable; see Shiryaev~\cite[pp. 187, 211]{Shi96}.
Serfozo~\cite[Lemma~3.2]{Ser82} established Fatou's lemma for a sequence
of measures converging vaguely on a locally compact metric space and for nonnegative functions.
Feinberg~et~al.~\cite[Theorem~4.2]{FKZTVP} provided Fatou's lemma for a sequence
of measures converging weakly and for functions bounded below by a sequence of functions satisfying a certain minorant condition, which is satisfied for nonnegative functions.
In this note we establish Fatou's lemma for a sequence of measures converging weakly and for functions  whose negative parts satisfy the uniform integrability condition.

Uniform integrability of a family of functions plays an important role in probability theory and analysis.  The relevant notion is asymptotic uniform integrability of a sequence of random variables~\cite[p. 17]{Vaa98}.  In this note we introduce the definitions of uniformly integrable (u.i.) and asymptotically uniformly integrable (a.u.i.) functions with respect to (w.r.t.) to a sequence of finite measures, and we show that these definitions are equivalent.  This equivalence provides alternative formulations and proofs to some facts that use uniform integrability or asymptotic uniform integrability assumptions. For the case of a single probability measure, this equivalence is established in Kartashov~\cite[p. 180]{Kar08}.

 Fatou's lemmas for weakly converging measures have significant applications to various areas and fields including stochastic processes~\cite{Chen, DEH, Has, LLZ}, statistics~\cite{TJ, Thorpe, GTSA}, control~\cite{Chu, FH, app2FKZ, FKZMOR,  VA}, game theory~\cite{app5in}, functional analysis~\cite{Gon}, optimization~\cite{XLS}, and electrical engineering~\cite{RWJS}. Our initial motivation in studying Fatou's lemma for variable probabilities  was caused by its usefulness for the proof  of the validity of optimality inequalities and the existence of stationary optimal policies for infinite-horizon, Borel-state, average-cost Markov decision process with noncompact action sets and unbounded costs \cite{app2FKZ}.  These results have significant applications to inventory control \cite{Fei,FL}.

Other versions of Fatou's lemmas for variable measures are also important for applications.  The recently discovered uniform Fatou lemma and uniform Lebesgue convergence theorems \cite{FKZ16} play the central role in establishing sufficient optimality conditions for partially observable Markov decision processes with Borel state and action spaces \cite{FKZMOR}.  Unlike the classic Fatou's lemma, which provides sufficient conditions for Fatou's inequality, the uniform Fatou lemma states necessary and sufficient condition for the uniform version of Fatou's inequality.  If all the functions are absolutely integrable, these necessary and sufficient conditions are more general than the conditions in the classic Fatou lemma.  One of two necessary conditions in the uniform Fatou lemma \cite[Theorem 2.1]{FKZ16} is that the sequence of negative parts of functions is uniformly integrable with respect to the measures; see  \eqref{eq:tv:ui} below.  The central result of this paper, Theorem 2.4, states that this condition is sufficient for the validity of Fatou's inequality for weakly converging measures. The examples in Section~\ref{sec:counterexa}   show  that this condition and the sufficient condition in Assumption~\ref{as:minorant}, which was introduced in \cite{FKZTVP},  do not imply each other. In particular, Theorem 2.4 is useful for studying Markov decision processes and stochastic games with cost functions unbounded from above and from below; see \cite{DPR,FJN, JN} where such problems are studied.


Fatou's lemma and Lebesgue's convergence theorems for probabilities are classic facts in probability theory.  However, their versions for finite measures are also important for probability theory and its applications.  This is the reason  this paper and \cite{FKZ16} study finite measures rather than probabilities. For example, the theory of optimization of Markov decision processes with multiple criteria is based on considering occupancy (also often called occupation) measures, which typically are not probability measures~\cite{Bor}. Another example is \cite{LLZ}, where Fatou's lemma for nonnegative  functions and finite measures is used.

  Though uniform integrability and asymptotic uniform integrability properties of a sequence of functions with respect to a sequence of finite measures are equivalent, it is typically easier to verify asymptotic uniform integrability.  This  is important for applications.  For this reason we provide in Section~\ref{sec:uiaui} alternative formulations of the uniform Fatou lemma and Lebesgue convergence theorem from \cite{FKZ16} and two classic facts important for applications. In these formulations  uniform integrability is substituted with asymptotic uniform integrability.

Section~\ref{sec:fatou}  of this paper provides definitions, describes the equivalence of uniform integrability and  asymptotic uniform integrability,  and states Fatou's lemma  and Lebesgue's  dominated convergence theorem for weakly converging measures.
In particular, Fatou's lemma is  formulated in Section~\ref{sec:fatou} for weakly converging measures   and for a.u.i. sequences of functions, which is equivalent to the assumption that the sequence of functions is u.i.
Section~\ref{sec:counterexa}  illustrates with the examples that  the uniform integrability condition stated in Theorem~\ref{thm:Fatou:w:ui} neither implies nor is implied by the minorant condition;   see Assumption~\ref{as:minorant} and Corollary~\ref{cor:Fatou:w:var} below. 
Example~\ref{eq:sw2aaa:varinf} demonstrates that $\limsup$ in inequalities \eqref{eq:sw2aaa:var} in Assumption~\ref{as:minorant} cannot be relaxed to $\liminf.$
By making use of the equivalence of uniform integrability and asymptotic uniform integrability,
Section~\ref{sec:uiaui} provides alternative formulations of uniform Fatou's lemma, uniform Lebesgue's  dominated convergence theorem, the Dunford-Pettis theorem, and Ball's fundamental theorem for Young measures.
Section~\ref{sec:proof thm ui ass} provides the proofs of Fatou's lemma for weakly converging measures  and relevant statements formulated in Section~\ref{sec:fatou}.

\section{ Main results}\label{sec:fatou} 

Let $(\SS,\SSig)$ be a measurable space, $\M (\SS)$ be the {\it family of all 
finite measures on} $(\SS,\SSig),$ and $\P(\SS)$ be the {\it family of all probability measures} on
$(\SS,\SSig).$ When $\SS$ is a topological space, as a rule, we consider
$\SSig := \B(\SS),$ where $\B(\SS)$ is the {\it Borel $\sigma$-field} on $\SS.$
Let $\R$ be the real line  and $\olR := [-\infty,+\infty].$ 
The notation $\h A$ denotes the indicator of the event $A.$

Throughout this paper, we deal with integrals of functions
that can take both positive and negative values.  An \textit{integral} $\int_\SS f (s)\mu(ds)$
of a measurable $\overline{\R}$-valued function $f$ on $\SS$ with
respect to a measure $\mu$ is \textit{defined} and equal to
\[\int_\SS f
(s)\mu(ds)=\int_\SS f^+ (s)\mu(ds)- \int_\SS f^- (s)\mu(ds),\]
if
\[
\min\{ \int_\SS f^+(s)\mu(ds),
\int_\SS f^-(s) \mu(ds)\}< \infty,
\]
where $f^+(s)=\max\{f(s),0\},$ $f^-(s)=-\min\{f(s),0\},$ $s\in \SS.$ All the
integrals in the assumptions of the theorems and
corollaries throughout this paper are assumed to be defined.

\begin{definition}\label{defi:unif integr}
The sequence of measurable $\olR$-valued functions $\{f_n\}_{n=1,2,\ldots}$ is called
\begin{itemize}
\item {\it uniformly integrable (u.i.) w.r.t. {{a} sequence of measures} $\{\mu_n\}_{n=1,2,\ldots}\subset \M (\SS)$} if
\begin{align}
		\lim_{K\to+\infty} \sup_{n=1,2,\ldots} \int_\SS |f_n (s)|
		\h \{ s\in\SS : |f_n (s)| \geq K\} \mu_n (ds) = 0 ;
		\label{eq:tv:ui}
	\end{align}
\item {\it asymptotically uniformly integrable (a.u.i.) w.r.t. {{a} sequence of measures} $\{\mu_n\}_{n=1,2,\ldots}\subset \M (\SS)$} if  
\begin{align}
		\lim_{K\to+\infty} \slim_{n\to\infty} \int_\SS |f_n (s)|
		\h \{ s\in\SS : |f_n (s)| \geq K\} \mu_n (ds) = 0.
		\label{eq:tv:ui:lim}
	\end{align}
\end{itemize}
\end{definition}
 We remark  that the limit as $K\to+\infty$ in \eqref{eq:tv:ui} (\eqref{eq:tv:ui:lim})
exists because the function
\begin{align}
K\mapsto &\sup_{n=1,2,\ldots} \big(\slim_{n\to\infty} \big) \int_\SS |f_n (s)|
		\h \{ s\in\SS : |f_n (s)| \geq K\} \mu_n (ds)
\end{align}
is nonincreasing in $K>0.$

If $\mu_n=\mu\in\M(\SS)$ for each $n=1,2,\ldots,$ then an (a.)u.i.~w.r.t. $\{\mu_n\}_{n=1,2,\ldots}$
sequence $\{f_n\}_{n=1,2,\ldots}$ is called (\textit{a.})\textit{u.i.}
For a single finite measure $\mu,$ the definition of an a.u.i.~sequence of functions
(random variables in the case of a probability measure $\mu$) coincides with the corresponding
definition broadly used in the literature; see, e.g., \cite[p. 17]{Vaa98}. Also, for a single fixed
finite measure, the definition of a u.i. sequence of functions is consistent with the classic
definition of a family $\cal H$ of u.i. functions.
We say that a function $f$ is (a.)u.i.~w.r.t.
$\{\mu_n\}_{n=1,2,\ldots}$ if the sequence $\{f,f,\ldots\}$ is  (a.)u.i.~w.r.t.
$\{\mu_n\}_{n=1,2,\ldots}.$ A function $f$ is  u.i.~w.r.t. a family  $ \mathcal{N} $
of measures if
\begin{align*}
	\lim_{K\to+\infty} \sup_{\mu\in  \mathcal{N}  }  \int_\SS |f (s)|
	\h \{ s\in\SS : |f (s)| \geq K\} \mu (ds) = 0.
\end{align*}

The following theorem states
the equivalence of the uniform and asymptotic uniform integrability properties introduced in Definition~\ref{defi:unif integr}.
The proof of Theorem~\ref{thm:uiCondEqui} is presented in the Appendix.
Several examples of applications of Theorem~\ref{thm:uiCondEqui} are provided in Section~\ref{sec:uiaui}. As mentioned in the Introduction, for $\mu_n=\mu$ with $\mu$ being a probability measure, $n=1,2,\ldots,$  Theorem~\ref{thm:uiCondEqui} is presented in Kartashov~\cite[p. 180]{Kar08}.

\begin{theorem}[{cp. Kartashov~\cite[p. 180]{Kar08}}]\label{thm:uiCondEqui}
Let $(\SS, \SSig)$ be a measurable space, $\{\mu_n\}_{n=1,2,\ldots}$ be a sequence of measures from $\M (\SS),$ and
$\{ f_n \}_{n=1,2,\ldots}$ be a sequence of measurable
$\olR$-valued functions on $\SS.$ Then there exists $ N=0,1,\ldots$ such that $\{f_{n+   N  }\}_{n=1,2,\ldots}$ is u.i.~w.r.t. $\{\mu_{n+  N  }\}_{n=1,2,\ldots}$ if and only if $\{f_n\}_{n=1,2,\ldots}$ is a.u.i.~w.r.t. $\{\mu_n\}_{n=1,2,\ldots}.$
\end{theorem}

We recall that Fatou's lemma claims that for a sequence of nonnegative measurable functions $\{f_n\}_{n=1,2,\ldots}$ defined on a measurable space $(\SS,\SSig)$  and for a measure $\mu$ on this space
\begin{equation}\label{eqfatouEF}
 \int_\SS \liminf_{n\to\infty} f_n(s)\mu(ds)\le\liminf_{n\to\infty} \int_\SS f_n(s)\mu(ds).
\end{equation}
Although a sequence of functions is u.i. if and only if it is a.u.i., in  many cases it is easier to verify that the sequence of functions is a.u.i. than that it is u.i.

\begin{definition}\label{def:wc}
	A sequence of measures $\{ \mu_n \}_{n=1,2,\ldots}$ on a metric space $\S$
	\textit{converges weakly} to a finite measure $\mu$ on $\S$
	if for each bounded continuous function $f$ on $\S$
\begin{align}
	\lim_{n\to\infty}\int_\S f(s) \mu_n (ds) = \int_\S f(s) \mu (ds) .
	\label{eq:def:weak}
\end{align}
\end{definition}

The following  theorem 
is the main result of this section. We provide the proof of this theorem in Section~\ref{sec:proof thm ui ass}.
\begin{theorem}[Fatou's lemma for weakly converging measures]\label{thm:Fatou:w:ui}
	Let $\S$ be a metric space, $\{\mu_n\}_{n=1,2,\ldots}$ be a sequence of measures on $\S$ converging weakly to $\mu\in \M (\S),$ and
	$\{ f_n \}_{n=1,2,\ldots}$ be a sequence of measurable
	$\olR$-valued functions on $\S$ such that $\{ f_n^- \}_{n=1,2,\ldots}$ is a.u.i.~w.r.t. $\{\mu_n\}_{n=1,2,\ldots}.$
Then
	\begin{align}\label{eq:lf}
		\int_\S \ilim\limits_{n\to\infty,\, s'\to s} f_n (s')\mu(ds)
		\le \ilim\limits_{n\to\infty}\int_\S f_n (s)\mu_n (ds).
	\end{align}
\end{theorem}
Consider the following assumption introduced in \cite{FKZTVP}, which is a sufficient condition for the validity of Fatou's lemma for weakly converging measures.
\begin{assumption}\label{as:minorant}
Let $\S$ be a metric space, $\{\mu_n\}_{n=1,2,\ldots}$ be a sequence of measures on $\S$ that  converges weakly to $\mu\in \M (\S),$ and
$\{ f_n, g_n \}_{n=1,2,\ldots}$ be a sequence of measurable
$\olR$-valued functions on $\SS$ such that $f_n(s)\ge g_n(s)$ for each $n=1,2,\ldots$ and $s\in\S,$	
and
	\begin{align}\label{eq:sw2aaa:var}
		-\infty<\int_\S\slim_{n\to\infty,s^\prime\to s} g_n(s^\prime)\mu(ds)
		\le\ilim_{n\to\infty}\int_\S g_n(s)\mu_n(ds).
	\end{align}
\end{assumption}

We note that Assumption~\ref{as:minorant} implies under certain conditions that the sequence of functions $\{ f_n^- \}_{n=1,2,\ldots}$ is u.i.~w.r.t. $\{\mu_n \}_{n=1,2,\ldots}$; see Theorem~\ref{thm:var->ui} below. In  general, these two conditions do not imply each other; see Examples~\ref{ex:ui>var} and \ref{ex:ui<var}. The following theorem, whose proof is provided in Section~\ref{sec:proof thm ui ass}, describes a sufficient condition when uniform integrability is more general than Assumption~\ref{as:minorant}. In addition, according to Example~\ref{ex:ui>var},  these two assumptions are not equivalent under the sufficient condition stated in Theorem~\ref{thm:var->ui}.

\begin{theorem}\label{thm:var->ui}
Let Assumption~\ref{as:minorant} hold. If the sequence of functions $\{ g_n \}_{n=1,2,\ldots}$
is uniformly bounded from above, then there exists $N=  0,1,\ldots $ such that $\{f_{n+N}^-\}_{n=1,2,\ldots}$ is u.i. w.r.t. $\{\mu_{  n+N  }\}_{n=1,2,\ldots}.$
\end{theorem}

For weakly converging probability measures, Fatou's lemma is introduced  in \cite{Ser82} and is generalized in \cite[Theorem~4.2]{FKZTVP}.
The following corollary extends \cite[Theorem~4.2]{FKZTVP} to finite measures. The proof of Corollary~\ref{cor:Fatou:w:var} is provided in Section~\ref{sec:proof thm ui ass}.
 Example~\ref{exa:to as25} demonstrates that Corollary~\ref{cor:Fatou:w:var} is incorrect if Assumption~\ref{as:minorant} is weakened by replacing $\limsup$  with $\liminf$ in formula \eqref{eq:sw2aaa:var}.
\begin{corollary}[Fatou's lemma for weakly converging measures; {cp.~\cite[Theorem~4.2]{FKZTVP}}]\label{cor:Fatou:w:var}
Inequality \eqref{eq:lf} holds under Assumption~\ref{as:minorant}.
\end{corollary}

The following corollary states Lebesgue's dominated convergence theorem for weakly converging measures. The similar statement
is provided in Serfozo~\cite[Theorem~3.5]{Ser82}  in the form of the necessary and sufficient condition for nonnegative functions and for locally compact spaces.  Though local compactness is not used in the proof of
\cite[Theorem~3.5]{Ser82}, there is a difference beween the cases of nonnegative and general functions.   If the functions can take both positive and negative values, the converse statement to Corollary~\ref{cor:leb:weak} does not hold.  This can be seen from Example~\ref{ex:ui<var}. 
%
\begin{corollary}[Lebesgue's  dominated convergence theorem for weakly converging measures; {cp.~\cite[Theorem~3.5]{Ser82}}]\label{cor:leb:weak}
Let $\S$ be a metric space, $\{\mu_n\}_{n=1,2,\ldots}$ be a sequence of measures on $\S$ that  converges weakly to
$\mu\in \M (\S),$ and $\{ f_n \}_{n=1,2,\ldots}$ be a sequence of measurable $\olR$-valued functions on $\S$ such that
$\lim_{n\to\infty,\, s'\to s} f_n (s')$ exists for $\mu$-a.e. $s\in\S.$ If $\{ f_n \}_{n=1,2,\ldots}$ is a.u.i.~w.r.t. $\{\mu_n\}_{n=1,2,\ldots},$ then

\begin{align}
	\label{eq:lth}\lim\limits_{n\to\infty}\int_\S  f_n (s)\mu_n (ds) = \int_\S \lim_{n\to\infty,\, s'\to s} f_n (s')\mu(ds) .
\end{align}
\end{corollary}

 \begin{proof} The corollary directly follows from Theorem~\ref{thm:Fatou:w:ui}, applied to the
sequences $\{ f_n \}_{n=1,2,\ldots}$ and  $\{- f_n \}_{n=1,2,\ldots}.$ \end{proof}

 The following assumption provides a sufficient condition for a sequence of measurable functions  $\{ f_n \}_{n=1,2,\ldots}$ to be u.i.~w.r.t. a sequence of finite measures $\{\mu_n\}_{n=1,2,\ldots}.$
 \begin{assumption}\label{as:majorant}
Let $\S$ be a metric space, $\{\mu_n\}_{n=1,2,\ldots}$ be a sequence of measures on $\S$ that  converges weakly to $\mu\in \M (\S),$ and
$\{ f_n, g_n \}_{n=1,2,\ldots}$ be a sequence of pairs of measurable
$\olR$-valued functions on $\SS$ such that $|f_n(s)|\le g_n(s)$ for each $n=1,2,\ldots$ and $s\in\S,$	
and
	\begin{equation}\label{eq:dcth suff ai1}
\slim_{n\to\infty}\int_\S g_n(s)\mu_n(ds)\le \int_\S\ilim_{n\to\infty,s^\prime\to s} g_n(s^\prime)\mu(ds)<+\infty.
\end{equation}
\end{assumption}
\begin{corollary}[Lebesgue's  dominated convergence theorem for weakly converging measures;  cp.{~\cite[Theorem~3.3]{Ser82}}]\label{cor:thleb:w:var}
	If Assumption~\ref{as:majorant} holds and $\lim_{n\to\infty,\, s'\to s} f_n (s')$ exists for $\mu$-a.e. $s\in\S,$ then equality \eqref{eq:lth} holds.
\end{corollary}
\begin{proof}
According to Theorem~\ref{thm:var->ui}, applied to $\mathtt{f}_n(s):=-|f_n(s)|$ and $\mathtt{g}_n(s):=-g_n(s),$ $n=1,2,\ldots,$ $s\in\S,$
Assumption~\ref{as:majorant} implies that $\{ f_n \}_{n=1,2,\ldots}$ is u.i.  w.r.t. $\{\mu_n\}_{n=1,2,\ldots}.$
In view of Theorem~\ref{thm:uiCondEqui},
the rest of the proof follows from Corollary~\ref{cor:leb:weak}.
\end{proof}

\section{Counterexamples}\label{sec:counterexa}
The following two examples illustrate that
uniform integrability of $\{ f_{n+N}^-\}_{n=1,2,\ldots}$ for some $N=0,1,\ldots$ neither implies nor is implied by Assumption~\ref{as:minorant}.
In Example~\ref{ex:ui>var}, $\{ f_n^-\}_{n=1,2,\ldots}$ is u.i.~w.r.t. $\{ \mu_n \}_{n=1,2,\ldots},$
but Assumption~\ref{as:minorant} does not hold.

\begin{example}\label{ex:ui>var}
{\rm
	Consider $\S := [0,1]$ endowed with the standard Euclidean metric and consider the probability measures
	\begin{align}
		\mu_n (C) := \int_C n\h \{ s\in[0,\frac{1}{n}] \}\nu(ds), \quad
		\mu (C) := \h \{ 0 \in C \}, \qquad C\in\B(\SS), \quad n=1,2,\ldots,
		\label{eq:ex:pm}
	\end{align}
	where $\nu$ is the Lebesgue measure on $[0,1].$ Then $\mu_n$ converges weakly to $\mu$ as $n\to\infty.$
	Let $f_n:\S\mapsto\R,$ $n=1,2,\ldots,$ be
	\begin{align}
		f_n (s) =
		\begin{cases}
			-i, & \text{if } s \in [\frac{1}{n}(1-2^{-(i-1)}), \frac{1}{n}(1- 2^{-i})), \quad i=1,2,\ldots; \\
			0, & \text{otherwise. }
		\end{cases}
		\label{eq:ex:ui>var:fn}
	\end{align}
	Then  $\int_\S f_n (s) \h \{ s\in\S : f_n (s) \leq -K\}
	\mu_n (ds) =  \sum_{i=\lceil K \rceil}^{\infty} \frac{-i}{2^i} =
	- \frac{\lceil K \rceil + 1}{2^{\lceil K \rceil-1}}$ for each $K > 0$ and for all $n=1,2,\ldots\ .$
	Since $\frac{\lceil K \rceil + 1}{2^{\lceil K \rceil-1}}\to 0$ as
	$K\to+\infty,$ the sequence $\{ f_n^-\}_{n=1,2,\ldots}$ is u.i.~w.r.t. $\{ \mu_n \}_{n=1,2,\ldots}.$
	Now, we show that Assumption~\ref{as:minorant} does not hold.
	Consider an arbitrary sequence of measurable functions	$\{ g_n \}_{n=1,2,\ldots}$
	such that $g_n (s) \leq f_n (s)$ for all $n=1,2,\ldots$ and for all $s\in\SS.$
	Let us prove that \eqref{eq:sw2aaa:var} does not hold.
	Assume, on the contrary, that \eqref{eq:sw2aaa:var} holds.
	Let $G:= \slim_{n\to\infty,s^\prime\to 0}g_n (s^\prime).$
	Since $g_n (s) \leq f_n (s) \leq 0,$ then $G\leq 0.$
	In view of \eqref{eq:ex:pm}, inequalities \eqref{eq:sw2aaa:var} become
    \begin{equation}\label{eq:sw2aaa:varEF}
	-\infty<G\leq \ilim\limits_{n\to\infty}\int_\SS g_n (s)\mu_n(ds).
    \end{equation}
	Note that, if \eqref{eq:sw2aaa:var} is true for $\{ g_n \}_{n=1,2,\ldots},$ then it is true
	for $\{ \tilde{g}_n\}$ such that $\tilde{g}_n (s) := g_n (s) - C,$ where $C\geq 0.$
	Therefore, we can select $\{g_n\}_{n=1,2,\ldots}$ such that
	$G\in \{-2,-3,\ldots\}.$ Then we show that $\ilim_{n\to\infty}\int_\S g_n (s) \mu_n (ds)<G.$
	Observe that the definition of $G$ implies
	\begin{align*}
		 \lim\limits_{m\to\infty} \sup_{n\geq m,s\in[0,\frac{1}{m}]} g_n (s)  \leq  G ;
	\end{align*}
	in fact, the equality takes place, but we do not need it.
	Then for every $\ee >0$ there exists $N(\ee) > 0$ such that $g_n(s) \leq G+ \ee$
	for all $ n\geq  N(\ee) $ and for all $s\in[0,\frac{1}{N(\ee)}].$
	Therefore, $g_n(s) \leq \min\{ G+ \ee, f_n(s) \}$ for all
	$s\in [0,\frac{1}{N(\ee)}]$ and  for all $ n\geq  N(\ee), $
	which implies
	\begin{equation}\label{eqGeps391}\int_\S g_n (s) \mu_n (ds) \leq \int_\S \min\{ G+ \ee, f_n(s) \} \mu_n (ds), \qquad  n\geq  N(\ee) . \end{equation}
	Let $\ee \in(0, 1).$     For $ n\geq  N(\ee)$
	\begin{align}\begin{split} & \int_\S \min\{ G+ \ee, f_n(s) \} \mu_n (ds)  =
		\int_{0}^{ \frac{1}{n}(1-{2^{G+1}}) }(G+ \ee)\mu_n (ds) +
		\int_{\frac{1}{n}(1-{2^{G+1}}) }^{\frac{1}{n}} f_n (s) \mu_n (ds) \\
		 = & ( G+ \ee )(1-{2^{G+1}})+ \sum_{i=-G}^{\infty} \frac{-i}{2^i}  = ( G+ \ee )(1-{2^{G+1}}) - \frac{-G+1}{2^{-G-1}} = G +\ee- (1+\ee) {2^{G+1}},
	\end{split}\label{eqGeps39}\end{align}
	where, as follows from \eqref{eq:ex:ui>var:fn}, the first equality holds because $f_n(s)\ge G+1 >G+\ee$ for $s\in[0,\frac{1}{n}(1-{2^{G+1}}))$ and $f_n(s)\le G<G+\ee$ for $s\in [\frac{1}{n}(1-{2^{G+1}}),\frac{1}{n})$.
	As follows from \eqref{eqGeps391} and \eqref{eqGeps39}, for every $\ee\in(0,1),$
	\begin{align*} \ilim_{n\to\infty}\int_\S g_n (s) \mu_n (ds)
		< G +\ee-  {2^{G+1}},
	\end{align*}	
	and, therefore,
	$\ilim_{n\to\infty}\int_\S g_n (s) \mu_n (ds) \le G - {2^{G+1}} < G,$
	which contradicts the second inequality in \eqref{eq:sw2aaa:varEF}. Hence, Assumption~\ref{as:minorant} does not hold.\hfill\qed
}
\end{example}

In addition, Kamihigashi's~\cite[Example 5.1]{Kam17}  of a sequence of  functions, which is not u.i.,  demonstrates that
Assumption~\ref{as:minorant} does not imply that $\{ f_{n  +N  }^-\}_{n=1,2,\ldots}$ is  u.i.~w.r.t. $\{ \mu_n \}_{n=1,2,\ldots}$ for some $N=0,1,\ldots\ .$ The following example is a slight modification of \cite[Example 5.1]{Kam17}.

\begin{example}[{cp. Kamihigashi~\cite[Example 5.1]{Kam17}}]\label{ex:ui<var}
{\rm
	Consider $ \S := [-1,1] $ endowed with the standard Euclidean metric. Let $\mu_n=\mu,$ $n=1,2,\ldots,$ be the
	Lebesgue measure on $\S$ and for 
$n=1,2,\ldots,$ $s\in\S$
	\begin{align*}
		f_n (s) =
		\begin{cases}
			-n, & \text{if } s \in [-\frac{1}{n},0); \\
			n, & \text{if } s\in (0,\frac{1}{n}]; \\
			0, & \text{otherwise}.
		\end{cases}
	\end{align*}
	Then  $\ilim_{n\to\infty}\int_\S f_n (s) \h \{ s\in\S : f_n (s) \leq -K\} \mu (ds)=-1$ for each $K > 0,$
	which implies that $\{ f_n^-\}_{n=1,2,\ldots}$ is not a.u.i. 
	Hence $\{ f_{n  +N }^-\}_{n=1,2,\ldots}$  is not u.i. for each $N=0,1,\ldots;$ 
    see Theorem~\ref{thm:uiCondEqui}.   For each $n=1,2,\ldots,$
	since $\int_\SS \slim_{n\to\infty,s^\prime\to s} f_n (s^\prime)\mu(ds) =\int_\S f_n (s) \mu (ds) = 0,$
 we have that	\eqref{eq:sw2aaa:var} holds  for $\mu_n=\mu$ with $g_n = f_n.$ \hfill\qed
}
\end{example}

The following example demonstrates that Corollary~\ref{cor:Fatou:w:var} fails if inequalities \eqref{eq:sw2aaa:var} are replaced in Assumption~2.5  with
	\begin{align}\label{eq:sw2aaa:varinf}
		-\infty<\int_\S\ilim_{n\to\infty,s^\prime\to s} g_n(s^\prime)\mu(ds)
		\le\ilim_{n\to\infty}\int_\S g_n(s)\mu_n(ds).
	\end{align}

\begin{example}{\rm (Inequalities \eqref{eq:sw2aaa:varinf} hold, but inequality \eqref{eq:lf} and the second inequality in \eqref{eq:sw2aaa:var} do not hold).}\label{exa:to as25}
{\rm

Let $\S:=[0,+\infty),$ $\mu_n(S)=\mu(S):=\int_S 2^{-s} ds,$ $S\in\B(\S),$ 
\[
f_n(s):=-2^n\h \{s\in[n,n+1)\},\quad {\rm and}\quad g_n(s):=f_n(s)-\frac{2^{s-1}}{\ln2}\sum_{k=0}^{2^n-1}\h\{s\in[\frac{2k}{2^n},\frac{2k+1}{2^n})\}
\]
for all $s\in[0,+\infty)$ and $n=1,2,\ldots.$
 Note that $f_n(s)\ge g_n(s)$ for all $s\in[0,+\infty)$ and $n=1,2,\ldots.$ Also,
\begin{equation}\label{eq:exaass25:1}
\ilim_{n\to\infty, s'\to s}f_n(s')=\slim_{n\to\infty, s'\to s}g_n(s')=0,\quad s\in[0,+\infty).
\end{equation}
Indeed, $\lim_{n\to\infty, s'\to s}f_n(s')\equiv 0$ because $f_n(s')=0$ for $s'\in [0,s+1),$ when $n\ge\lfloor s\rfloor +2$ and $s\in [0,+\infty),$ where $\lfloor a\rfloor $ is the integer part of the real number $a\in\R.$ Since $g_n(s)\le 0$ for all $s\in[0,+\infty)$ and $n=1,2,\ldots,$ the second equality in \eqref{eq:exaass25:1} follows from
\begin{equation}\label{eq:exaass25:2}
0\ge \slim_{n\to\infty, s'\to s}g_n(s')\ge \lim_{n\to\infty} g_n(\frac{2\lfloor 2^{n-1}s\rfloor +\frac32}{2^{n}})=0,\quad s\in[0,+\infty),
\end{equation}
where the second inequality in \eqref{eq:exaass25:2} holds because $s-\frac{1}{2^{n+1}}< \frac{2\lfloor 2^{n-1}s\rfloor +\frac32}{2^{n}}\le s+ \frac{3}{2^{n+1}}$ for each $n=1,2,\ldots,$ and $\lim_{n\to\infty}(s-\frac{1}{2^{n+1}})=\lim_{n\to\infty}(s+\frac{3}{2^{n+1}})= s;$ and the equality in \eqref{eq:exaass25:2} holds because $\frac{2\lfloor 2^{n-1}s\rfloor+\frac32}{2^{n}}\notin \cup_{k=0}^{2^n-1} [\frac{2k}{2^n},\frac{2k+1}{2^n})$ for each $n=1,2,\ldots.$
Therefore, equalities \eqref{eq:exaass25:1} hold.

We observe that
\begin{equation}\label{eq:exaass25:3}
\ilim_{n\to\infty, s'\to s}g_n(s')=-\frac{2^{s-1}}{\ln2}\h\{s\in[0,2]\},\quad s\in[0,+\infty).
\end{equation}
Indeed, since $g_n(s)\ge f_n(s)-\frac{2^{s-1}}{\ln2}\h\{s\in[0,2]\}$ for each $s\in[0,+\infty)$ and $n=1,2,\ldots,$ and the function $s\mapsto\ilim_{n\to\infty, s'\to s}g_n(s') $ is lower semi-continuous, equality \eqref{eq:exaass25:3} follows from
\begin{equation}\label{eq:exaass25:4}
\ilim_{n\to\infty, s'\to s}g_n(s')=\lim_{n\to\infty} g_n(\frac{2\lfloor2^{n-1}s\rfloor+\frac12}{2^{n}})=-\frac{2^{s-1}}{\ln2},\quad s\in[0,2),
\end{equation}
where the first equality in \eqref{eq:exaass25:4} holds because $s-\frac{3}{2^{n+1}}< \frac{2\lfloor 2^{n-1}s\rfloor +\frac12}{2^{n}}\le s+ \frac{1}{2^{n+1}}$ for each $n=1,2,\ldots,$ and $\lim_{n\to\infty}(s-\frac{3}{2^{n+1}})=\lim_{n\to\infty}(s+\frac{1}{2^{n+1}})= s;$  and the second equality in \eqref{eq:exaass25:4} holds because $\frac{2\lfloor 2^{n-1}s\rfloor +\frac12}{2^{n}}\in \cup_{k=0}^{2^n-1} [\frac{2k}{2^n},\frac{2k+1}{2^n})$ for  $n\ge \max\{1,\lfloor-\log_2(2-s)\rfloor\}.$
Therefore, equality \eqref{eq:exaass25:3} holds.

 Equality \eqref{eq:exaass25:3} implies
\begin{equation}\label{eq:exaass25:4.5}
\int_0^\infty \ilim_{n\to\infty, s'\to s}g_n(s') \mu(ds)=-\int_0^2\frac{1}{2\ln2}ds =-\frac{1}{\ln2}.
\end{equation}

For each $n=1,2,\ldots,$
\begin{equation}\label{eq:exaass25:5}
\int_0^\infty f_n(s) \mu(ds)=-2^n\int_n^{n+1}2^{-s}ds=\frac{2^n}{\ln2}(2^{-n-1}-2^{-n})=-\frac{1}{2\ln2}
\end{equation}
and
%
\begin{equation}\label{eq:exaass25:6}
\int_0^\infty g_n(s) \mu(ds)=\int_0^\infty f_n(s) \mu(ds)-\frac{1}{2\ln2}\sum_{k=0}^{2^n-1}\frac{1}{2^n}= -\frac{1}{\ln2}.
\end{equation}

Inequalities \eqref{eq:sw2aaa:varinf} hold because, according to \eqref{eq:exaass25:4.5} and \eqref{eq:exaass25:6},
\[
-\infty<-\frac{1}{\ln2}=\int_\S\ilim_{n\to\infty,s^\prime\to s} g_n(s^\prime)\mu(ds)
		\le\ilim_{n\to\infty}\int_\S g_n(s)\mu_n(ds)=-\frac{1}{\ln2}.\]
However,  inequality \eqref{eq:lf} does not hold because, according to \eqref{eq:exaass25:5} and \eqref{eq:exaass25:1},
\[
-\frac{1}{2\ln2}=\ilim\limits_{n\to\infty}\int_\S f_n (s)\mu_n (ds)< \int_\S \ilim\limits_{n\to\infty,\, s'\to s} f_n (s')\mu(ds)=0.
\]
The second inequality in \eqref{eq:sw2aaa:var} does not hold either because, according to \eqref{eq:exaass25:6} and \eqref{eq:exaass25:1},
\[
-\frac{1}{\ln2}=\ilim_{n\to\infty}\int_\S g_n(s)\mu_n(ds)<\int_\S\slim_{n\to\infty,s^\prime\to s} g_n(s^\prime)\mu(ds)=0.
\]

Therefore, inequalities \eqref{eq:sw2aaa:varinf} hold, but inequality \eqref{eq:lf} and the second inequality in \eqref{eq:sw2aaa:var} do not hold. \hfill\qed}
\end{example}

\section{Examples of applications of Theorem~\ref{thm:uiCondEqui}}
\label{sec:uiaui}

This section provides 
 examples of applications of Theorem~\ref{thm:uiCondEqui}.  The usefulness of these applications is that it is typically easier to verify asymptotic u.i. w.r.t. to a sequence of  measures than u.i. 



\subsection{Uniform Fatou's lemma and uniform Lebesgue's  dominated
convergence theorem for measures converging in total variation}
\label{subsec:Lebtotal}

{
The following statements  are
\cite[Theorem~2.1 and Corollary 2.9]{FKZ16} with conditions~(ii)  replacing the conditions that $\{ f^-_n\}_{n=1,2,\ldots}$  and  $\{ f_n\}_{n=1,2,\ldots}$ are u.i.~w.r.t. $\{ \mu_n \}_{n=1,2,\ldots}$  respectively.  
As explained in \cite{FKZ16}, inequality \eqref{eq:ufl:tv} is stronger than the inequality in Fatou's lemma, and the sufficient condition in Proposition~\ref{thm:UFL:FKZ} can be viewed as the uniform version of Fatou's lemma.
Since the convergence in \eqref{eq:ufl:tvLeb} is a uniform version of convergence of integrals, the sufficient condition in Proposition~\ref{thm:UFL:FKZ111LEB} can be viewed as the uniform version of Lebesgue's  dominated convergence theorem.  
}

\begin{proposition}[Uniform Fatou's lemma; cp. {\cite[Theorem 2.1]{FKZ16}}]\label{thm:UFL:FKZ}
Let $(\SS, \SSig)$ be a measurable space, a sequence of measures
$\{\mu_n\}_{n=1,2,\ldots}$ from $\M (\SS)$ converge
in total variation to a measure $\mu\in\M(\SS),$ and
$\{ f,f_n\}_{n=1,2,\ldots}$ be a sequence of measurable $\olR$-valued functions on $\SS.$
Assume that $f\in L^1 (\SS;\mu)$ and  $f_n\in L^1 (\SS;\mu_n)$ for each $n=1,2,\ldots\ .$  Then the inequality
\begin{align}
	\ilim_{n\to\infty} \inf_{C\in\SSig} \Big( \int_C f_n (s)\mu_n (ds) - \int_C f(s) \mu (ds) \Big) \geq 0
	\label{eq:ufl:tv}
\end{align}
 takes place if and only if the following two statements hold:
\begin{itemize}
	\item[{\rm (i)}] for each $\ee > 0$
	\begin{align*}
		\mu (\{ s\in\SS : f_n (s) \leq f(s) - \ee \}) \to 0 \text{ as }  n\to \infty;
		\end{align*}
	\item[{\rm (ii)}] $\{f^-_n\}_{n=1,2,\ldots}$ is a.u.i.~w.r.t. $\{\mu_n\}_{n=1,2,\ldots}.$
\end{itemize}
\end{proposition}
\begin{proof}
The theorem follows from \cite[Theorem~2.1]{FKZ16} and Theorem~\ref{thm:uiCondEqui}.
\end{proof}

 We notice that, since $\emptyset\in\Sigma,$  the left-hand side of \eqref{eq:ufl:tv} is nonpositive.  Therefore, inequality  \eqref{eq:ufl:tv} takes place if and only if it holds in the form of the equality. Since the left hand side of \eqref{eq:ufl:tv} is the lower limit of a sequence of nonpositive numbers, the lower limit in \eqref{eq:ufl:tv} is the limit.

\begin{proposition}[Uniform Lebesgue's  dominated convergence theorem; cp. 
{\cite[Corollary~2.9]{FKZ16}}]\label{thm:UFL:FKZ111LEB}
Let $(\SS, \SSig)$ be a measurable space,
$\{\mu_n\}_{n=1,2,\ldots}$ be a sequence of measures from $\M (\SS)$ converging
in total variation to a measure $\mu\in\M(\SS),$ and
$\{ f,f_n \}_{n=1,2,\ldots}$ be a sequence of measurable $\olR$-valued functions on $\SS.$
Assume that $f\in L^1 (\SS;\mu)$ and $f_n\in L^1 (\SS;\mu_n)$ for each $n=1,2,\ldots\ .$ Then
\begin{align}
	\lim_{n\to\infty} \sup_{C\in\SSig} \Big| \int_C f_n (s)\mu_n (ds) - \int_C f(s) \mu (ds) \Big| = 0
	\label{eq:ufl:tvLeb}
\end{align}
 if and only if the following two statements hold:
\begin{itemize}
	\item[{\rm (i)}] $\{f_n\}_{n=1,2,\ldots}$ converge to $f$ in measure $\mu;$
	\item[{\rm (ii)}] $\{f_n\}_{n=1,2,\ldots}$ is a.u.i.~w.r.t. $\{\mu_n\}_{n=1,2,\ldots}.$ 
\end{itemize}
\end{proposition}
\begin{proof}
The theorem follows from \cite[Corollary~2.9]{FKZ16} and Theorem~\ref{thm:uiCondEqui}.
\end{proof}

\subsection{On Dunford-Pettis theorem}\label{subsec:dp}
As follows from Eberlein-\v{S}mulian theorem,
the Dunford-Pettis theorem implies that the sequence $\{ f_n \}_{n=1,2,\ldots}\subset L^1(\SS;\mu)$
has a weakly convergent subsequence $\{f_{n_k}\}_{k=1,2,\ldots}$ to $f\in L^1(\SS;\mu)$ in $L^1(\SS;\mu)$ if and only if
$\{ f_n \}_{n=1,2,\ldots}$ is u.i.; see, for example, Albiac and Kalton~\cite[Theorem~5.2.9, p.~109]{Alb}, Bogachev~\cite[Theorem~4.7.18, p.~285]{Bog}, Diestel~\cite[p.~93]{Dis}, Edwards~\cite[Theorem~4.21.2, p.~274]{Edw}, Meyer~\cite[T23, p.~20]{Mey66}, Treves~\cite[Theorem~46.1, p.~471]{Tre}, and Wojtaszczyk~\cite[Theorem~12, p.~137]{Woj}.

The main result of this subsection has the following formulation.

\begin{proposition}\label{thm:Dunf-Pett}
Let $(\SS, \SSig)$ be a measurable space, $\mu\in \M (\SS),$ and
$\{ f_n \}_{n=1,2,\ldots}\subset L^1(\SS;\mu)$ be a sequence of measurable
$\olR$-valued functions on $\SS.$ Then the following statements are equivalent:
\begin{itemize}
\item[(i)] there exists $\{f_{n_k}\}_{k=1,2,\ldots}\subset\{ f_n \}_{n=1,2,\ldots}$
such that $f_{n_k}\to f$ weakly in $L^1(\SS;\mu)$ for some $f\in L^1(\SS;\mu);$
\item[(ii)] there exists $ N=0,1,\ldots  $ such that $\{f_{n+  N  }\}_{n=1,2,\ldots}$ is u.i.;
\item[(iii)] $\{f_n\}_{n=1,2,\ldots}$ is a.u.i.
\end{itemize}
\end{proposition}

\begin{proof}
In view of Eberlein-\v{S}mulian theorem, statements (i) and (ii) are equivalent due to the Dunford-Pettis theorem.
The equivalence of statements~(ii) and (iii) directly follows from Theorem~\ref{thm:uiCondEqui}.
\end{proof}

\subsection{The fundamental theorem for Young measures}\label{subsec:Ball1989}

In this subsection we provide an equivalent formulation of the fundamental theorem
for 
Young measures from Ball~\cite{Ball1989}. Let $n,m=1,2,\ldots,$ $\Omega\subset\R^n$ be Lebesgue measurable, $C\subset \R^m$ be closed. Let ${\rm meas}$ denote the Lebesgue measure on $\R^n.$ Consider Banach spaces $L^1(\Omega)$ and $L^\infty(\Omega)$ of all integrable and essentially bounded functions on $\Omega$ respectively, endowed with the standard norms.

\begin{proposition}\label{thm:Ball1989}
Let $z^{(j)}:\Omega\mapsto \R^m,$ $j=1,2,\ldots,$ be a sequence of Lebesgue measurable functions satisfying $z^{(j)}(\,\cdot\,)\to  C$ in measure as $j\to\infty,$ that is, for every
neighbourhood $U$ of $ C$ in $\R^m$
\[
\lim_{j\to\infty} {\rm meas} \{x\in\Omega \, :\, z^{(j)}(x)\notin U\}= 0.
\]
Then there exists a subsequence $\{z^{(j_k)}\}_{k=1,2,\ldots}$ of $\{z^{(j)}\}_{j=1,2,\ldots}$ and a family $(\nu_x)$, $x\in\Omega,$ of positive measures on $\R^m,$ depending measurably on $x,$ such that
\begin{itemize}
\item[{\rm(i)}] $\|\nu_x\|_M:=\int_{\R^m} d\nu_x\le 1$ for a.e. $x\in\Omega;$
\item[{\rm(ii)}] ${\rm supp\,} \nu_x\subset C$ for a.e. $x\in\Omega;$ and
\item[{\rm(iii)}] $f(z^{(j_k)})\to\langle \nu_x,f\rangle=\int_{\R^m}f(\lambda) d\nu_x(\lambda)$ weakly star in $L^\infty(\Omega)$ for each continuous function $f:\R^m\mapsto \R$ satisfying $\lim_{|\lambda|\to \infty} f(\lambda)=0.$
\end{itemize}
Suppose further that $\{z^{(j_k)}\}_{k=1,2,\ldots}$ satisfies the asymptotic boundedness condition
\begin{equation}\label{eq:Ball1989:1}
\lim_{K\to+\infty} \slim_{k\to\infty}\, {\rm meas} \{x\in \Omega\cap B_R\,:\, |z^{(j_k)}(x)|\ge K \}=0,
\end{equation}
for every $R>0,$ where $B_R=  B_R(\bar{0})  $  is a ball of radius R and center $\bar{0}$ in the Euclidean n-space $\R^n.$ Then  $\|\nu_x\|_M=1$ for a.e. $x\in\Omega$ (that is, $\nu_x$ is a probability measure), and given any measurable subset $A$ of $\Omega$
\begin{equation}\label{eq:Ball1989:2}
f(z^{(j_k)})\to\langle \nu_x,f\rangle\quad\mbox{ weakly in }L^1(A)
\end{equation}
for any continuous function $f:\R^m\mapsto \R$ such that $\{f(z^{(j_k)})\}_{k=1,2,\ldots}$ is sequentially
weakly relatively compact in $L^1(A).$
\end{proposition}

\begin{remark}\label{rem:Ball1989:1}
{\rm
Ball~\cite[Theorem]{Ball1989} is Proposition~\ref{thm:Ball1989} with \eqref{eq:Ball1989:1} replaced with
\begin{equation}\label{eq:Ball1989:3}
\lim_{ K\to+\infty} \sup_{k=1,2,\ldots}\, {\rm meas} \{x\in \Omega\cap B_R\,:\, |z^{(j_k)}(x)|\ge K \}=0 { .} 
\end{equation}
}
\end{remark}

\begin{remark}\label{rem:Ball1989:2}
{\rm
Condition~(\ref{eq:Ball1989:1}) is equivalent to the following one: given any $R>0$ there exists a continuous function $g_R:[0,+\infty)\mapsto \R,$ with $\lim_{t\to\infty}g_R(t)=+\infty,$ such that
\begin{equation}\label{eq:Ball1989:4}
\slim_{k\to\infty}\int_{\Omega\cap B_R} g_R(|z^{(j_k)}(x)|)dx <\infty;
\end{equation}
see 
Ball~\cite[Remark~1]{Ball1989}.
 Without loss of generality, we can also replace limit superior with limit interior in \eqref{eq:Ball1989:4} if Condition~(\ref{eq:Ball1989:1}) is replaced with the condition in this remark.
}
\end{remark}

\begin{remark}\label{rem:Ball1989:3}
{\rm
If $A$ is bounded in Proposition~\ref{thm:Ball1989}, then Proposition~\ref{thm:Dunf-Pett} implies that the condition that
$\{f(z^{(j_k)})\}_{k=1,2,\ldots}$ is sequentially
weakly relatively compact in $L^1(A)$ is satisfied if and only if $\{f(z^{(j_k)})\}_{k=1,2,\ldots}$ is a.u.i.
}
\end{remark}

\begin{proof}[Proof of Proposition~\ref{thm:Ball1989}]
All statements follow from Theorem from Ball~\cite{Ball1989} and Theorem~\ref{thm:uiCondEqui}, being applied to $\SS:=\Omega\cap B_R$ endowed with the Lebesgue $\sigma$-algebra $\SSig$ on $\Omega\cap B_R,$
\begin{align*}
	f_k(s) & :=|z^{(j_k)}(s)|\h\{s\in\SS\,:\, |z^{(j_k)}(s)|\ge 1\}, \\
	\mu_k(S) & :=\int_S\frac{ds}{|z^{(j_k)}(s)|\h\{s\in\SS\,:\, |z^{(j_k)}(s)|\ge 1\}},
\end{align*}
for each $S\in\SSig$ and sufficiently large $k{\ge}1.$  Indeed, since
\[
\begin{aligned}
&\int_\SS |f_k (s)|
		\h \{ s\in\SS : |f_k (s)| \geq  K\} \mu_k (ds)={\rm meas} \{x\in \Omega\cap B_R\,:\, |z^{(j_k)}(x)|\ge K\},
\end{aligned}
\]
for each $ K >1$ and sufficiently large $k{\ge}1,$ we have that Theorem~\ref{thm:uiCondEqui} implies that \eqref{eq:Ball1989:1} and \eqref{eq:Ball1989:3} are equivalent, and, therefore, the statements of Theorem from Ball \cite{Ball1989} and Proposition~\ref{thm:Ball1989} are equivalent.
\end{proof}	

\section{Proofs of Theorems~\ref{thm:Fatou:w:ui}, \ref{thm:var->ui} and Corollary~\ref{cor:Fatou:w:var}}\label{sec:proof thm ui ass}

This section contains the proofs of Theorems~\ref{thm:Fatou:w:ui}, \ref{thm:var->ui} and Corollary~\ref{cor:Fatou:w:var}.

\begin{proof}[Proof of Theorem~\ref{thm:Fatou:w:ui}]
Let us fix an arbitrary $K > 0.$ Then
\begin{equation}	
\begin{aligned}
	 \ilim_{n\to\infty} \int_\S f_n (s) &\mu_n (ds) \\
	  \geq\ilim_{n\to\infty}& \int_\S f_n (s) \h \{ s\in\S: f_n (s) > -K \} \mu_n (ds)\\  +& \ilim_{n\to\infty} \int_\S f_n (s) \h \{ s\in\S: f_n (s) \leq -K \} \mu_n (ds).
	\label{eq:lbd:ui}
	\end{aligned}
\end{equation}

The following inequality holds:
\begin{equation}\label{eq:finK}
\ilim_{n\to\infty} \int_\S f_n (s) \h \{ s\in\S: f_n (s) > -K \} \mu_n (ds)\ge \int_\S \ilim_{n\to\infty,s^\prime\to s}f_n(s^\prime)\mu (ds).
\end{equation}
Indeed, if $\mu (\S)=0,$ then
\[
\ilim_{n\to\infty} \int_\S f_n (s) \h \{ s\in\S: f_n (s) > -K \} \mu_n (ds)\ge -K\lim_{n\to\infty}\mu_n(\S)=0= \int_\S \ilim_{n\to\infty,s^\prime\to s}f_n(s^\prime)\mu (ds),
\]
where the equalities hold because $\mu_n(\S)\to\mu(\S)=0$ as $n\to\infty.$ Otherwise, if $\mu (\S)>0,$ then \cite[Theorem~4.2]{FKZTVP}, applied to $\{\tilde{f}_n\}_{n=1,2,\ldots}:=\{f_{n+N}\}_{n=1,2,\ldots},$ $\tilde{g}_n\equiv-K,$ $\tilde{\mu}_n (C) := \frac{\mu_{n+N}(C)}{\mu_{n+N}(\S)}$ and $\tilde{\mu}(C)=\frac{\mu(C)}{\mu(\S)},$
for each $n=1,2,\ldots$ and $C\in\B(\S),$ where $N=1,2,\ldots$ is sufficiently large,
implies
\begin{equation}\label{eq:finK1}
\begin{aligned}
\ilim_{n\to\infty} &\int_\S f_n (s) \h \{ s\in\S: f_n (s) > -K \} \mu_n (ds)\\
&\ge \int_\S \ilim_{n\to\infty,s^\prime\to s}f_n (s^\prime) \h \{ s^\prime\in\S: f_n (s^\prime) > -K \}\mu (ds).
\end{aligned}
\end{equation}
Here we note that $\{\tilde{\mu}_n\}_{n=1,2,\ldots}\subset\P(\S)$ converges weakly to
$\tilde{\mu}\in\P(\S).$ We remark also that
\begin{equation}\label{eq:finK2}
f_n(s)\h \{ s\in\S: f_n(s)>-K \} \geq f_n(s)
\end{equation}
for all $s\in\S$ because $K>0.$ Thus, \eqref{eq:finK} follows from \eqref{eq:finK1} and \eqref{eq:finK2}.

Inequalities \eqref{eq:lbd:ui} and \eqref{eq:finK} imply
\[
\begin{aligned}
	 \ilim_{n\to\infty} \int_\S f_n (s) &\mu_n (ds)
	  \geq \int_\S\ilim_{n\to\infty,s^\prime\to s}f_n(s^\prime)\mu (ds)\\  & + \lim_{K\to+\infty}\ilim_{n\to\infty} \int_\S f_n (s) \h \{ s\in\S: f_n (s) \leq -K \} \mu_n (ds),
	\end{aligned}
\]
which is equivalent to \eqref{eq:lf} because $\{ f_n^- \}_{n=1,2,\ldots}$ is a.u.i.~w.r.t. $\{\mu_n\}_{n=1,2,\ldots}.$
\end{proof}

\begin{proof}[Proof of Theorem~\ref{thm:var->ui}]
Let Assumption~\ref{as:minorant} hold. According to Theorem~\ref{thm:uiCondEqui},
it is sufficient to prove that $\{f_{n}^-\}_{n=1,2,\ldots}$ is a.u.i.~w.r.t. $\{\mu_n\}_{n=1,2,\ldots},$ that is,
\begin{align}
		\lim_{K\to+\infty} \ilim_{n\to\infty} \int_\SS f_n (s)
		\h \{ s\in\SS : f_n (s) \leq -K\} \mu_n (ds) = 0 .
		\label{eq:tv:laui:lim}
	\end{align}

Let us prove  (\ref{eq:tv:laui:lim}). Indeed,
since $f_n(s)\ge g_n(s),$  
\[
\h\{s\in\S\,:\, f_n(s)\le -K\}\le \h\{s\in\S\,:\, g_n(s)\le- K\},
 \]
for all $n=1,2,\ldots ,$ $K>0,$ and $s\in\S.$ Therefore,
\begin{equation}\label{eq:maj:eq:1}
\begin{aligned}
\lim_{K\to+\infty} &\ilim_{n\to\infty} \int_\SS f_n (s)
		\h \{ s\in\SS : f_n (s) \leq -K\} \mu_n (ds)\\ &\ge \lim_{K\to+\infty} \ilim_{n\to\infty} \int_\SS g_n (s)
		\h \{ s\in\SS : g_n (s) \leq -K\} \mu_n (ds).
\end{aligned}
\end{equation}
Inequalities \eqref{eq:sw2aaa:var} imply
\begin{equation}\label{eq:maj:eq:2}
\begin{aligned}
\lim_{K\to+\infty} \ilim_{n\to\infty}& \int_\SS g_n (s)
		\h \{ s\in\SS : g_n (s) \leq -K\} \mu_n (ds) \\
\ge \int_\S&\slim_{n\to\infty,s^\prime\to s} g_n(s^\prime)\mu(ds)+ \lim_{K\to+\infty} \ilim_{n\to\infty} \int_\SS -g_n (s)
		\h \{ s\in\SS : g_n (s)> -K\} \mu_n (ds).
\end{aligned}
\end{equation}
Since the functions $\{ g_n \}_{n=1,2,\ldots}$
are bounded from above by the same constant,  Theorem~\ref{thm:Fatou:w:ui}, applied to the sequence of the functions $\{\mathtt{f}_n\}_{n=1,2,\ldots},$ which are uniformly bounded from below, where $\mathtt{f}_n(s):=-g_n (s)
		\h \{ s\in\SS : g_n (s)> -K\},$ $s\in\S,$ $n=1,2,\ldots,$ implies
\begin{equation}\label{eq:maj:eq:3}
\begin{aligned}
\ilim_{n\to\infty} &\int_\SS -g_n (s)
		\h \{ s\in\SS : g_n (s)> -K\} \mu_n (ds)\\ &\ge
-\int_\S\slim_{n\to\infty,s^\prime\to s} g_n (s^\prime)
		\h \{ s^\prime\in\SS : g_n (s^\prime)>- K\}\mu(ds)
\end{aligned}
\end{equation}
for each $K>0.$ If for each $s\in\S$
\begin{equation}\label{eq:maj:eq:4}
\slim_{n\to\infty,s^\prime\to s} g_n (s^\prime)
		\h \{ s^\prime\in\SS : g_n (s^\prime)> -K\}\downarrow\slim_{n\to\infty,s^\prime\to s} g_n (s^\prime)\quad {\rm as}\quad K\uparrow+\infty,
		\end{equation}
 then \eqref{eq:maj:eq:1}--\eqref{eq:maj:eq:3} directly imply \eqref{eq:tv:laui:lim}, that is, $\{ f_n^-\}_{n=1,2,\ldots}$ is a.u.i.~w.r.t. $\{\mu_n\}_{n=1,2,\ldots}.$

Let us prove \eqref{eq:maj:eq:4}. Since for each $s^\prime\in \S$ and $n=1,2,\ldots$
\[
g_n (s^\prime)
		\h \{ s^\prime\in\SS : g_n (s^\prime) >- K\}\downarrow g_n (s^\prime)\quad {\rm as} \quad K\uparrow+\infty,
\]
 we have that
\[
\sup_{m\ge n,s^\prime\in B_\delta(s)}  g_m (s^\prime)
		\h \{ s^\prime\in\SS :  g_m (s^\prime) >- K\}\downarrow\sup_{m\ge n,s^\prime\in B_\delta(s)}  g_m (s^\prime)\quad {\rm as}\quad K\uparrow+\infty
\]
for each $n=1,2,\ldots$ and $\delta>0,$ where $B_\delta(s)$ is the ball in the metric space $\S$ of radius $\delta$ centered at $s.$ Therefore,
\[
\inf_{n\ge1,\delta>0}\sup_{m\ge n,s^\prime\in B_\delta(s)}  g_m (s^\prime)
		\h \{ s^\prime\in\SS : g_m (s^\prime) >- K\} \downarrow \inf_{n\ge1,\delta>0}\sup_{m\ge n,s^\prime\in B_\delta(s)}  g_m (s^\prime)\quad  {\rm as} \quad K\uparrow+\infty,
\]
that is, \eqref{eq:maj:eq:4} holds for each $s\in\S.$ Thus, $\{ f_n^-\}_{n=1,2,\ldots}$ is u.i.~w.r.t. $\{\mu_n\}_{n=1,2,\ldots}.$
\end{proof}

\begin{proof}[Proof of Corollary~\ref{cor:Fatou:w:var}]

Theorem~\ref{thm:Fatou:w:ui}, applied to the sequence $\{f_n-g_n\}_{n=1,2,\ldots}$ of nonnegative functions, implies
\begin{equation}
\begin{aligned}
&\int_\S \ilim\limits_{n\to\infty,\, s'\to s} f_n (s')\mu(ds)-\int_\S \slim\limits_{n\to\infty,\, s'\to s} g_n (s')\mu(ds)  \\
& \le \int_\S \ilim\limits_{n\to\infty,\, s'\to s} (f_n (s')-g_n(s'))\mu(ds)  \le \ilim\limits_{n\to\infty}\int_\S (f_n (s)-g_n (s)) \mu_n (ds)  \label{endprc27} \\
&\le \ilim\limits_{n\to\infty}\int_\S f_n (s)\mu_n (ds)- \ilim\limits_{n\to\infty} \int_\S g_n (s)\mu_n (ds),
	\end{aligned}
\end{equation}
where the first and the third inequalities follow from the basic properties of infimums and supremums.  Inequality \eqref{eq:lf}  follows from \eqref{endprc27} and Assumption~\ref{as:minorant}.
\end{proof}

{\bf Acknowledgements.} The research of the first and the third author was supported by NSF Grant CMMI-1636193.  The authors thank G.~M.~Shevchenko for bringing their attention to M.~V.~Kartashov's book~\cite{Kar08}.

\renewcommand\refname{\center{\normalfont{\small{REFERENCES}}}}


\begin{appendices}

\section*{Appendix: Proof of Theorem~\ref{thm:uiCondEqui}}
\addtocounter{section}{1}
\setcounter{equation}{0}
%
This appendix contains the proof of Theorem~\ref{thm:uiCondEqui}.  This proof is close to the proof of the similar result in Kartashov~\cite[p. 180]{Kar08} for the case of a single probability measure.  The main reason for providing this appendix is that reference~\cite{Kar08} may not be available for the majority of the readers. The authors of this note have not been aware about \cite{Kar08} for long time.
After the first version of this note, which contained a direct proof of Theorem~\ref{thm:uiCondEqui}, was posted in arxiv.org, Professor G.M.~Shevchenko informed the authors about the book by M.V.~Kartashov~\cite{Kar08}.


\begin{lemma}[{Kartashov~\cite[p.~134]{Kar08}}]\label{lm:Kar}
	Consider a sequence of real-valued functions $\ee_n (K),$ where $K>0,$ such that
	\begin{itemize}
		\item[{\rm (a)}] $\ee_n (K)\downarrow0$ as $K\to+\infty$ for each $n=1,2,\ldots;$ and
		\item[{\rm (b)}] $\lim_{K\to+\infty}\slim_{n\to\infty} \ee_n (K) = 0.$	
	\end{itemize}
	Then {\rm (c)} $\lim_{K\to+\infty}\sup_{n=1,2,\ldots} \ee_n (K) = 0.$	
\end{lemma}

\begin{proof}  On the contrary
 assume that {\rm (c)} does not hold.  In this case the limit in {\rm (c)} is  equal to  $\delta$ for some $\delta>0.$  Observe that $\sup_{n=1,2,\ldots} \ee_n (K)\ge {\delta}$ for all $K>0$ since this function does not increase in $K.$   For each $m=1,2,\ldots$ there is a natural number $n_m$ such that $\ee_{n_m} (m) \geq \delta/2 .$ If the sequence $\{n_m\}_{m=1,2,\ldots}$ is bounded, then it is possible to choose $n_m = k$ for some $k\in\{1,2,\ldots\}$ and for an infinite subset of integer numbers $m.$ Therefore, $\ee_k (m) \geq \delta/2 $ for these numbers, which contradicts assumption {\rm (a)}. For $m> K,$ as follows from monotonicity in {\rm (a)},  $\ee_n (K) \geq \ee_n (m),$ which implies
\begin{align}\label{econKAER}
	\slim_{n\to\infty} \ee_n (K) \geq \slim_{m\to\infty} \ee_{n_m} (K) \geq \slim_{m\to\infty} \ee_{n_m} (m)\geq \delta/2 >0,
\end{align}
where $K>0$ is an arbitrary real number, and \eqref{econKAER} contradicts condition {\rm (b)}.
\end{proof}

\begin{proof}[Proof of Theorem~\ref{thm:uiCondEqui}]
The uniform integrability w.r.t. $\{\mu_{n+  N  }\}_{n=1,2,\ldots}$ of the sequence $\{f_{n+  N  }\}_{n=1,2,\ldots}$  for some $  N =0,1,\ldots  $ implies the asymptotic uniform integrability w.r.t. $\{\mu_n\}_{n=1,2,\ldots}$ of the sequence $\{ f_n \}_{n=1,2,\ldots}.$  

Vice versa, let $\{f_n\}_{n=1,2,\ldots}$ is a.u.i.~w.r.t. $\{\mu_n\}_{n=1,2,\ldots}.$ Then there exists $  N  =0,1,\ldots  $
such that $f_n\in L^1(\SS;\mu_n)$ for each $n=  N +1,N+2,\ldots\ .  $ Indeed, if there exists a subsequence $\{f_{n_k}\}_{k=1,2,\ldots}$ of $\{f_n\}_{n=1,2,\ldots}$ such that $\int_\SS |f_{n_k}(s)| \mu_{n_k} (ds) = \infty,$ then $\int_\SS |f_{n_k}(s)|\h\{s\in\SS:|f_{n_k}|\geq K \} \mu_{n_k} (ds) = \infty$ for each $K>0$ and $k=1,2,\ldots\ .$ Therefore,
\begin{align*}
\lim_{K\to+\infty} \slim_{n\to\infty} \int_\SS |f_n (s)|
		\h \{ s\in\SS : |f_n (s)| \geq K\} \mu_n (ds) = \infty ,
\end{align*}
which contradicts the assumption that $\{f_n\}_{n=1,2,\ldots}$ is a.u.i.~w.r.t. $\{\mu_n\}_{n=1,2,\ldots}.$

Consider $\ee_n (K) := \int_\SS |f_{  n+N } (s)|\h \{ s\in\SS : |f_{  n+N  } (s)| \geq K\} \mu_{  n+N  } (ds),$ $n=1,2,\ldots\ .$ Since $f_{  n+N  }\in L^1(\SS;\mu_{  n+N  })$ for each $n=1,2,\ldots,$
condition (a) in Lemma~\ref{lm:Kar} holds. Condition (b) in Lemma~\ref{lm:Kar} holds because $\{f_n\}_{n=1,2,\ldots}$ is a.u.i.~w.r.t. $\{\mu_n\}_{n=1,2,\ldots}.$ Then Lemma~\ref{lm:Kar} implies the validity of \eqref{eq:tv:ui} for $\{f_{n+  N  }\}_{n=1,2,\ldots}$ w.r.t. $\{\mu_{n+  N  }\}_{n=1,2,\ldots}.$ 
\end{proof}

\end{appendices}

\end{document}